\documentclass[a4paper]{amsart}

\usepackage[british,english]{babel} 
\usepackage{graphics,color,pgf}
\usepackage{epsfig}
\usepackage[ansinew]{inputenc}
\usepackage[all]{xy}
\usepackage{hyperref}
\usepackage{amssymb, amsmath,amsthm, mathtools, amscd}
\usepackage{mathrsfs}
\usepackage{stmaryrd}
\usepackage{cancel} %strike out
\usepackage{tikz}
\usetikzlibrary{matrix}

\theoremstyle{plain}
 \newtheorem{thm}{Theorem}
 \newtheorem{prop}[thm]{Proposition}
 \newtheorem{lem}[thm]{Lemma}
 \newtheorem{cor}[thm]{Corollary}
 \newtheorem*{op*}{Open Problem} 
 \newtheorem*{frthmA*}{Theorem A} 
 \newtheorem*{frcorB*}{Corollary B} 
  \newtheorem*{frthmAp*}{Theorem A'} 
   
\theoremstyle{definition}

\theoremstyle{remark}
 \newtheorem*{rem}{Remark}

\newcommand{\R}{\mathbb{R}}

\renewcommand{\geq}{\geqslant}
\renewcommand{\setminus}{\smallsetminus}
\title[Kernels in measurable cohomology]{Kernels in measurable cohomology for transitive actions}

\author[Michelle Bucher]{%\bfseries 
Michelle Bucher} 
\address{Universit\'e de Gen\`eve}
\email{Michelle.Bucher-Karlsson@unige.ch}

\author[Alessio Savini]{%\bfseries 
Alessio Savini} %% Please write ful names, avoid initials
\address{University of Milano-Bicocca%   \\ %\hfill (Revised  00 00 201?)\\
}
\email{alessio.savini@unimib.it}

\thanks{Supported by the Swiss National Science Foundation} %% optional
\begin{document}

\begin{abstract}
Given a connected semisimple Lie group $G$, Monod \cite{Monod} has recently proved that the measurable cohomology of the $G$-action $H^*_m(G \curvearrowright G/P)$ on the Furstenberg boundary $G/P$, where $P$ is a minimal parabolic subgroup, maps surjectively on the measurable cohomology of $G$ through the evaluation on a fixed basepoint. Additionally, the kernel of this map depends entirely on the invariant cohomology of a maximal split torus. In this paper we show a similar result for a fixed subgroup $L<P$ such that the stabilizer of almost every pair of points in $G/L$ is compact. More precisely, we show that the cohomology of the $G$-action $H^p_m(G \curvearrowright G/L)$ maps surjectively onto $H^p_m(G)$ with a kernel isomorphic to $H^{p-1}_m(L)$. Examples of such groups are given either by any term of the derived series of the unipotent radical $N$ of $P$ or by a maximal split torus $A$. We conclude the paper by computing explicitly some cocycles on quotients of $\mathrm{SL}(2,\mathbb{K})$ for $\mathbb{K}=\mathbb{R}, \mathbb{C}$. 
\end{abstract}

\maketitle

\section{Introduction}

Given a semisimple Lie group $G$, its measurable cohomology $H^*_m(G)$ is defined via the cocomplex of $G$-invariant (classes of) measurable functions $(L^0(G^{*+1})^G,d^*)$, where $d^*$ is the usual homogeneous coboundary operator. Similarly, its continuous cohomology is defined by looking at the cocomplex of $G$-invariant continuous functions $(C_c(G^{\ast+1})^G,d^\ast)$. Recently, Austin and Moore \cite{AuMoo} have proved that the natural inclusion of continuous cochains into measurable ones induces an isomorphism in cohomology. The result actually holds in a far more general setting, namely for any $G$ locally compact second countable group and for every Fr\'echet $G$-module as coefficients. 

The cohomology $H^\ast_m(G)$ can be viewed as the measurable cohomology of the action of $G$ on itself. This point of view allows to extend the definition of measurable cohomology to any homogeneous $G$-space, endowed with its invariant measure class. For a closed subgroup $L<G$, it is sufficient to define $H^*_m(G \curvearrowright G/L)$ via the cocomplex of $G$-invariant measurable functions $(L^0((G/L)^{*+1})^G,d^*)$. There exists a natural evaluation map 
\begin{equation} \label{eq evaluation map}
\mathrm{ev}:H^*_m(G \curvearrowright G/L) \rightarrow H^*_m(G)
\end{equation}
obtained by fixing any basepoint $x \in G/L$ and considering the cochain map
$$
\mathrm{ev}_x:L^0((G/L)^{*+1})^G \rightarrow L^0(G^{*+1})^G, \ \ (\mathrm{ev}_x)(f)(g_0,\ldots,g_*):=f(g_0x,\ldots,g_*x) . 
$$
In the attempt to give a more geometric flavour to the classes lying in $H^*_m(G)$, one could ask for which closed subgroup $L<G$ the evaluation map is an isomorphism or, at least, it is surjective. 

In a recent paper Monod \cite{Monod} focused his attention on the case of the Furstenberg-Poisson boundary $G/P$, where $P<G$ is any miminal parabolic subgroup. He proved that the evaluation is indeed surjective, but surprisingly it is not an isomorphism, unless the rank of $G$ is equal to one. In fact, in the higher rank case a non-trivial kernel appears when the cohomological degree is sufficiently small and such defect can be completely characterized in terms of the cohomology of a maximal split torus $A<P$. More precisely, let us denote by $\mathfrak{a}$ the Lie algebra associated to $A$ and by $w_0$ the longest element in the Weyl group, which acts on $\mathfrak{a}$ via the adjoint representation. We have the following 

\begin{thm}\cite[Theorem B]{Monod}\label{Thm Monod} Let $G$ be a connected semisimple Lie group with finite center. The evaluation map 
$$ H^*_m(G\curvearrowright G/P)\longrightarrow H^*_m(G)$$
is surjective and its kernel 
$$NH^*_m(G\curvearrowright G/P):= \mathrm{Ker}(H^*_m(G\curvearrowright G/P)\longrightarrow H^*_m(G))$$
fits into a short exact sequence
\begin{equation}\label{SESP}(\wedge^{k-2}\mathfrak{a}^*)^{w_0} \overset{i}\longrightarrow NH^k_m(G\curvearrowright G/P) \longrightarrow  (\wedge^{k-1}\mathfrak{a}^*)^{w_0},\end{equation}
for $k \geq 3$, and, for $k=2$, there is an isomorphism 
$$
NH^2_m(G \curvearrowright G/P) \cong (\mathfrak{a}^\ast)^{w_0}.
$$
\end{thm}

More recently the authors \cite{BucSavAlt} have shown that the short exact sequence \eqref{SESP} corresponds to the decomposition into non-alternating and alternating cochains on the boundary.

Monod's proof does not rely on the usual characterization of group cohomology in terms of resolutions, but it exploits the Eckmann-Shapiro isomorphism \cite[Theorem 6]{Moore} together with spectral sequences. A similar approach already appeared in a work by Bloch \cite{Bloch} to study the measurable cohomology of $\mathrm{SL}(2,\mathbb{C})$ and it was later developed by Pieters \cite{Pieters} in the case of $\mathrm{Isom}^+(\mathbb{H}^n)$. 

The advantage of Monod's proof is that one can follow his spectral sequence to obtain explicit representatives of the cohomology classes lying in the kernel of the evaluation map. Following the latter strategy, the authors \cite{BucSav} obtained explicit cocycles in low degrees for either the products of isometry groups of hyperbolic spaces or $\mathrm{SL}(3,\mathbb{K})$, where $\mathbb{K}=\mathbb{R},\mathbb{C}$. In fact, as was pointed out to us by Elisha Falbel, our $\mathrm{SL}(3,\mathbb{K})$-cocycles on the flag space are already cohomologically trivial when viewed as cocycles on the space of affine flags. Since affine flags surject onto flags, the previous remark shows in a different way that our cocycles must lie in the kernel of the evaluation map for $G/P$. This observation motivated the present study of different homogeneous spaces for which the evaluation map \eqref{eq evaluation map} is still surjective but some non-trivial kernel appears in low degrees. For instance, Goncharov \cite{Goncharov} exhibited an unbounded exotic cocycle lying in the kernel of the $\mathrm{SL}(3,\mathbb{C})$-action on $\mathbb{P}^2(\mathbb{C})$. The latter corresponds to the choice of the homogeneous space $\mathrm{SL}(3,\mathbb{C})/Q$, where $Q$ is a maximal parabolic subgroup. In this paper we will focus our attention on spaces of the form $G/L$, for some particular $L$ contained in a minimal parabolic subgroup.

\begin{thm}\label{Main thm} Let $G$ be a connected semisimple Lie group with finite center. Let $P<G$ be a minimal parabolic subgroup. Consider $L<P$ a closed subgroup such that the stabilizer of almost every pair of points in $G/L$ is compact. Then the evaluation map 
$$
H^*_m(G \curvearrowright G/L) \longrightarrow H^*_m(G)
$$
is surjective and its kernel 
$$
NH^*_m(G \curvearrowright G/L):=\mathrm{Ker}(H^*_m(G \curvearrowright G/L) \longrightarrow H^*_m(G))
$$
is isomorphic to 
$$
NH^p_m(G \curvearrowright G/L) \cong H^{p-1}_m(L),
$$
for $p \geq 2$. 

Moreover this construction is compatible with inclusions: more precisely, given $L_0<L_1<P$ two closed subgroups as above, the projection $G/L_0 \rightarrow G/L_1$ induces a commutative diagram 
\begin{equation}\label{diagram restriction identity}
\xymatrix{
H^{p-1}_m(L_1) \ar[rr] \ar[d]^{\mathrm{res}} && H^p_m(G \curvearrowright G/L_1) \ar[rr] \ar[d] && H^p_m(G) \ar[d]^{=} \\
H^{p-1}_m(L_0) \ar[rr] && H^p_m(G \curvearrowright G/L_0) \ar[rr] && H^p_m(G) \ .
}
\end{equation}
where $\mathrm{res}$ is the restriction map. 
\end{thm}

The strategy of the proof of Theorem \ref{Main thm} is analogous to the one given by Monod. In fact we construct a spectral sequence abutting to zero and whose first page contains the cohomology of $G$ and $L$ on the first and the second columns, respectively, and the cocomplex of measurable $G$-invariant functions on $G/L$ on the first row. The explicit computation of the differentials and the convergence to zero lead to the desired result. The stability statement is obtained by a simple comparison argument between the spectral sequences obtained by choosing different groups $L_0$ and $L_1$. 

Applying Theorem \ref{Main thm} to any element of the derived series associated to the unipotent radical $N<P$, we obtain:  

\begin{cor}\label{Thm N}
Let $G$ be a connected semisimple Lie group and let $N<P$ be the unipotent radical of a minimal parabolic subgroup $P$. We define inductively $N_{k+1}:=[N_k,N_k]$ as the commutator subgroup of $N_k$. 
Then the stabilizer of almost every pair in $G/N_k$ is compact. As a consequence, we have a short exact sequence
\begin{equation}\label{SES N}
\xymatrix{
0 \ar[r] & H^{p-1}_m(N_k) \ar[r]^{\hspace{-15pt} i} & H^p_m(G \curvearrowright G/N_k) \ar[r] & H^p_m(G) \ar[r] & 0,
}
\end{equation}
for $p \geq 2$. 
\end{cor}

For instance, when $G=\mathrm{SL}(3,\mathbb{R})$ we can choose as $P$ the subgroup of upper triangular matrices and its unipotent radical $N=\mathcal{H}_3(\mathbb{R})$ is the (real) Heisenberg group. Corollary \ref{Thm N} implies that the cohomology of $\mathcal{H}_3(\mathbb{R})$ injects in the cohomology of the $\mathrm{SL}(3,\mathbb{R})$-action on the affine flags (and the statement extends to every $n \geq 4$). 

Choosing as $L=A<P$ a maximal split torus, we get: 

\begin{cor}\label{Thm A}
Let $G$ be a connected semisimple Lie group and let $A<P$ a maximal split torus. Then the stabilizer of almost every pair in $G/A$ is compact. As a consequence, we have a short exact sequence
\begin{equation}\label{SES A}
\xymatrix{
0 \ar[r] & H^{p-1}_m(A) \ar[r]^{\hspace{-15pt} i} & H^p_m(G \curvearrowright G/A) \ar[r] & H^p_m(G) \ar[r] & 0,
}
\end{equation}
for $p \geq 2$. 
\end{cor}

\subsection*{An explicit example: $\mathrm{SL}(2,\mathbb{K})$} Following the spectral sequence defined in the proof of Theorem \ref{Main thm}, we will exhibit some explicit cocycles lying in the kernel of the evaluation map when $G=\mathrm{SL}(2,\mathbb{K})$, for $\mathbb{K}=\mathbb{R},\mathbb{C}$. If we define $P$ as the group of upper triangular matrices, we can choose as $L=N$ the unipotent radical of $P$ corresponding to unipotent matrices. In this context $G/N$ is diffeomorphic to the punctured plane $\mathbb{K}^2 \setminus \{0\}$ and, since $N \cong \mathbb{K}$, we have a non-trivial kernel contributing to the cohomology group $H^2_m(G \curvearrowright \mathbb{K}^2 \setminus \{ 0 \})$. 

\begin{thm}\label{Thm SL2 N}
Let $G=\mathrm{SL}(2,\mathbb{K})$ with $\mathbb{K}$ equal to either $\mathbb{R}$ or $\mathbb{C}$. Consider the group $N$ of upper unipotent triangular matrices with Lie algebra $\mathfrak{n} \cong \mathbb{K}$. The inclusion 
$$
i:H^1_m(N) \rightarrow H^2_m(G \curvearrowright (\mathbb{K}^2 \setminus \{0\})) 
$$
of the short exact sequence \ref{SES N} is given by sending a $\mathbb{R}$-linear form $\alpha_{\mathfrak{n}} \in \mathfrak{n}^*$ to the measurable almost everywhere defined $G$-invariant cocycle 
$$
\begin{array}{rcl}
i(\alpha_{\mathfrak{n}}):  (\mathbb{K}^2 \setminus \{0\})^3 &\longrightarrow &\mathbb{R}\\
(v_0,v_1,v_2)&\longmapsto & \sum_{i < j} (-1)^{i+j} \left[ \alpha_{\mathfrak{n}}\left(\langle  \frac{v_i}{d_{v_i,v_j}} , \frac{v_j}{\|v_j\|^2} \rangle \right)+\alpha_{\mathfrak{n}}\left( \langle \frac{v_j}{d_{v_i,v_j}} , \frac{v_i}{\|v_i\|^2} \rangle \right)\right],
\end{array}
$$
where $\langle \cdot, \cdot \rangle$ is either the scalar product, for $\mathbb{K}=\mathbb{R}$, or the Hermitian product, for $\mathbb{K}=\mathbb{C}$, and $d_{v_i,v_j}$ is the determinant of the matrix with columns $(v_i|v_j)$.  
\end{thm}

Then we consider the case of a maximal split torus $A$ for $\mathrm{SL}(2,\mathbb{R})$. In this case we fix as $A$ the subgroup of diagonal matrices with positive entries and the quotient $G/A$ is diffeomorphic to the subset of distinct ordered pairs in $\mathbb{P}^1(\mathbb{R})^2$. Since $A \cong \mathbb{R}$, also in this context we have a non trivial kernel appearing in $H^2_m(G \curvearrowright \mathbb{P}^1(\mathbb{R})^2)$. 

\begin{thm}\label{Thm SL2 A}
Let $G=\mathrm{SL}(2,\mathbb{R})$ and consider the torus $A$ of diagonal matrices with positive entries, with Lie algebra $\mathfrak{a} \cong \mathbb{R}$. The inclusion 
$$
i:H^1_m(A) \rightarrow H^2_m(G \curvearrowright \mathbb{P}^1(\mathbb{R})^2) 
$$
of the short exact sequence \ref{SES A} is given by sending a $\mathbb{R}$-linear form $\alpha_{\mathfrak{a}} \in \mathfrak{a}^*$ to the measurable almost everywhere defined $G$-invariant cocycle
$$
\begin{array}{rcl}
i(\alpha_{\mathfrak{a}}):  (\mathbb{P}^1(\mathbb{R})^2)^3 &\longrightarrow &\mathbb{R}\\
((x_1,x_2),(y_1,y_2),(z_1,z_2))&\longmapsto & \frac{1}{2}\alpha_{\mathfrak{a}}\left( \log|[x_2,y_1,y_2,z_1]-1| \right),
\end{array}
$$
where $[ \ , \ , \ ,  \ ]$ is the usual cross ratio. 
\end{thm}
 
\subsection*{Plan of the paper}
In Section \ref{sec main thm} we prove Theorem \ref{Main thm} via the use of spectral sequences. In Section \ref{sec derived series torus} we show that the main theorem holds for the cases of the derived series associated to the unipotent radical of a minimal parabolic subgroup (Corollary \ref{Thm N}) and for the maximal split torus (Corollary \ref{Thm A}). We conclude by an explicit computation on the group $\mathrm{SL}(2,\mathbb{K})$ for $\mathbb{K}=\mathbb{R},\mathbb{C}$ in Section \ref{sec unipotent radical SL2} and Section \ref{sec maximal split torus SL2}. 

\section{Proof of Theorem \ref{Main thm}}\label{sec main thm}

We are going to follow the line of Monod's proof of \cite[Theorem B]{Monod}. Let $G$ be a connected semisimple Lie group with finite center. Let $P$ be a minimal parabolic subgroup and consider $L<P$ such that the stabilizer of almost every pair of points in $G/L$ is compact. We define the bicomplex 
$$
C^{p,q}:=L^0(G^{p+1},L^0((G/L)^q))^G \cong L^0(G^{p+1} \times (G/L)^q)^G,
$$ 
with two differentials given by
$$
d^\uparrow:C^{p,q} \rightarrow C^{p+1,q}, \ \ d^\rightarrow:C^{p,q} \rightarrow C^{p,q+1}.
$$
The first differential is simply the usual homogeneous differential on the $G$-variable, whereas the second one coincides with the homogeneous differential on $G/L$ multiplied by $(-1)^{p+1}$. The multiplicative constant ensures that the differentials commute, that is $d^\uparrow d^\rightarrow = d^\rightarrow d^\uparrow$. 

The bicomplex that we introduced leads naturally to two different spectral sequences. The first one has as first page
\begin{equation}\label{eq first spectral sequence}
{}^IE_1^{p,q}:=(H^q(C^{p,\ast},d^\rightarrow),d_1=d^\uparrow). 
\end{equation}

\begin{prop}\label{prop collapse}
Let ${}^IE_1^{p,q}$ be the first page of the spectral sequence defined by Equation \eqref{eq first spectral sequence}. Then the spectral sequence degenerates immediately to zero, that is ${}^IE_1^{p,q}=0$ for every $p,q \geq 0$. 
\end{prop}

\begin{proof}
The first page of the spectral sequence is determined by taking the right differential $d^\rightarrow$ in the bicomplex $C^{p,q}$. More precisely, it is given by the cohomology of the cocomplex 
$$
L^0(G^{p+1},L^0((G/L)^{q-1}))^G \rightarrow L^0(G^{p+1},L^0((G/L)^q))^G \rightarrow L^0(G^{p+1},L^0((G/L)^{q+1}))^G.
$$
Such cocomplex is isomorphic to its \emph{inhomogeneous} variant obtained by getting rid of the $G$-invariance and deleting one $G$-variable, that is
\begin{equation}\label{inhomogenous complex} 
L^0(G^{p},L^0((G/L)^{q-1})) \rightarrow L^0(G^{p},L^0((G/L)^q)) \rightarrow L^0(G^{p},L^0((G/L)^{q+1})).
\end{equation}
Such isomorphism preserves the differential on $G/L$, which is still the homogeneous one suitably weighted with a sign depending on $p$. In virtue of \cite[Lemma 2.4]{Monod} the cocomplex 
\begin{equation}\label{resolution G/L}
0 \rightarrow \R \rightarrow L^0(G/L) \rightarrow L^0((G/L)^2) \rightarrow 
\end{equation}
is exact. Since the cohomology of the cocomplex of Equation \ref{inhomogenous complex} is obtained by the one of Equation \ref{resolution G/L} by applying the functor $L^0(G^{p}, \ \cdot \ )$, and the latter is an exact functor by \cite[Lemma 2.1]{Monod}, the statement follows. 
\end{proof}

Since ${}^IE_1^{p,q}$ vanishes, from now on we are going to focus on the second spectral sequence, namely
$$
{}^{II}E_1^{p,q}:=(H^p(C^{\ast,q},d^\uparrow),d_1=d^\rightarrow).
$$

The $(p,q)$-entry of the first page is given by 
$$
{}^{II}E_1^{p,q}=H^p_m(G,L^0((G/L)^q)). 
$$
We are going to compute explicitly all the differentials in degree one. Before starting we want to point out a problem related to the cohomological restriction induced by the inclusion $L \rightarrow G$. Since $L$ is a closed subgroup of $P$, it has measure zero with respect to the Haar measure on $G$. As a consequence, it would be meaningless to restrict an element of $L^0(G^p)^G$ to $L$. To overcome this difficulty, we will realize the cohomology of $L$ as the cohomology of the cocomplex $(L^0(G^{*+1})^L,d^*)$. Indeed the functor that assigns to every Polish $L$-module $A$ the cohomology of the cocomplex $(L^0(G^{\ast+1},A)^L,d^\ast)$ satisfies the characterizing proprerties given by Moore \cite[Pag. 16]{Moore}, and hence its computes the measurable cohomology of $L$ with coefficients in $A$ \cite[Theorem 2]{Moore}.

\begin{lem}\label{vanishing differential}
For $p>1$, the first differential
$$
d_1:H^p_m(G) \rightarrow H^p_m(G,L^0(G/L)) 
$$
induced by the inclusion of coefficients is trivial.
\end{lem}

\begin{proof}
It is immediate to verify that the differential $d_1$ is conjugated via the Eckmann-Shapiro isomorphism \cite[Theorem 6]{Moore} to the restriction from $G$ to $L$. As noticed by Monod \cite[Proof of Proposition 4.1]{Monod}, the Eckmann-Shapiro induction can be realized by the isomorphism given by
$$
\alpha: L^0(G^{p+1})^L \rightarrow L^0(G^{p+1},L^0(G/L))^G, \ \ \alpha(f)(g_0,\ldots,g_p)(gL):=f(g^{-1}g_0,\ldots,g^{-1}g_p). 
$$

We can consider the following commutative diagram 
\begin{equation}\label{diagram restriction}
\xymatrix{
L^0(G^{p+1})^G \ar[rr] \ar[d] && L^0(G^{p+1},L^0(G/L))^G \ar[d] \\
L^0(G^{p+1})^L \ar[rr]^{\hspace{-15pt}\alpha} && L^0(G^{p+1},L^0(G/L))^G,
}
\end{equation}
where the top map is the change of coefficients induced by the inclusion $\mathbb{R} \rightarrow L^0(G/L)$, the left arrow is the restriction map, the bottom arrow is the induction map in the Eckmann-Shapiro isomorphism and the right map is the identity. 

By passing to the cohomology of Diagram \ref{diagram restriction} we get
$$
\xymatrix{
H^p_m(G) \ar[rr]^{d_1} \ar[d]^{\mathrm{res}} && H^p_m(G,L^0(G/L)) \ar[d]^{=} \\
H^p_m(L) \ar[rr] && H^p_m(G,L^0(G/L)).
}
$$
Since the bottom map is the Eckmann-Shapiro isomorphism, it is sufficient to show that the restriction map $\mathrm{res}$ is zero. Since $L<P$ we get that the restriction map from $G$ to $L$ factors to the restriction from $G$ to $P$, that is we a commutative diagram
$$
\xymatrix{
H^p_m(G) \ar[rd]_{\mathrm{res}_P} \ar[rr]^{\mathrm{res}_L} && H^p_m(L)  \\ 
& H^p_m(P) \ar[ur] ,&
}
$$
where we wrote $\mathrm{res}_P$ and $\mathrm{res}_L$ to distinguish the different restrictions. By either \cite[Corollary 3]{Wienhard} or \cite[Corollary 3.2]{Monod} the restriction map $\mathrm{res}_P$ is zero for every $p>1$ and the claim follows. 
\end{proof}

Our next step is to prove that the remaining columns in the first page ${}^{II}E_1^{p,q}$ vanish except for the elements in the first row.

\begin{prop}\label{vanishing columns}
The measurable cohomology group $H^p_m(G,L^0((G/L)^q))$ vanishes for $p>1$ and $q \geq 2$. 
\end{prop}

\begin{proof}
Since the orbits of the $G$-action on $(G/L)^q$ are locally closed, the action is smooth in the sense of Zimmer by \cite[Theorem 2.1.14]{Zimmer}.

By assumption we know that the stabilizer of almost every pair in $G/L$ is compact. As a consequence, the stabilizer of almost every $q$-tuple of points in $G/L$ is compact. This means that, for almost every orbit, the smoothness of the action determines an identification between that orbit and a quotient of the form $G/K_0$, for some compact group $K_0<G$. By Cartan's Fixed Point Theorem, such a group $K_0$ must be contained in some maximal compact subgroup of $G$. As a consequence there are at most countably many $G$-conjugacy classes of such compact groups. 

By the smoothness of the action \cite[Theorem A.7]{Zimmer}, there exists a Borel measurable section with respect to the projection $(G/L)^q \rightarrow G \backslash (G/L)^q$. We first discard a measure zero subset $\Delta \subset (G/L)^q$ where the stabilizers are not compact. We can decompose $(G/L)^q \setminus \Delta$ into disjoint union of measurable $G$-invariant subsets $Y_j$, and on each $Y_j$ the stabilizer of a $q$-tuple is conjugated to $K_j$. This fact, together with the existence of the Borel section, implies that there is a Borel $G$-isomorphism between $(G/L)^q \setminus \Delta$ and a disjoint countable union of $G$-spaces $G/K_j \times X_j$, where $K_j$ is compact and $X_j$ is a measure space with trivial $G$-action which is the quotient of $Y_j$. Thanks to the previous isomorphism, the measurable cohomology $H^p_m(G,L^0((G/L)^q))$ is the product of the cohomology groups $H^p_m(G,L^0(G/K_j \times X_j))$. Using Fubini's Theorem and the Eckmann-Shapiro isomorphism we obtain 
$$
H^p_m(G,L^0(G/K_j \times X_j)) \cong H^p_m(G,L^0(G/K_j,L^0(X_j))) \cong H^p_m(K_j,L^0(X_j)). 
$$
The compactness of $K_j$, together with \cite[Lemma 2.3]{Monod}, forces the vanishing of the group $H^p_m(K_j,L^0(X_j))$, as desired. 
\end{proof}

We are finally ready to prove the main theorem of the section.

\begin{proof}[Proof of Theorem \ref{Main thm}] 
We describe the first page ${}^{II}E^{p,q}_1$. The first column is given by the cohomology of $G$, that is ${}^{II}E_1^{p,0} \cong H^p_m(G)$ for $p \geq 1$. The second column is isomorphic to the cohomology of $L$ by the Eckmann-Shapiro isomorphism, namely ${}^{II}E_1^{p,1} \cong H^p_m(L)$. Moreover the differential $d_1$ from the first column to the second one vanishes identically when $p \geq 1$ by Lemma \ref{vanishing differential}. 

The remaining columns in the first page vanish for $p>1$ by Proposition \ref{vanishing columns}, since ${}^{II}E_1^{p,q} \cong H^p_m(G,L^0((G/L)^q))$ with $q \geq 2$. We are only left with the first row which corresponds to the $G$-invariants of the coefficient $G$-modules $L^0((G/L)^q)^G$ and $d_1$ boils down to the usual homogeneous differential. The first page is depicted in Figure \ref{Page 1}.

\begin{figure}[!h]
\centering
\begin{tikzpicture}
  \matrix (m) [matrix of math nodes,
             nodes in empty cells,
             nodes={minimum width=9ex,
                    minimum height=9ex,
                    outer sep=-3pt},
             column sep=2ex, row sep=-1ex,
             text centered,anchor=center]{
         p      &         &          &          & \\
          \cdots & \cdots & \cdots & \cdots   & \cdots \\
          3    & \ H^3_m(G) \  & \ H^3_m(L) \ &\   0 & \ \cdots & \\
          2    &  \ H^2_m(G) \ &\  H^2_m(L) \  & \  0 & \  \cdots &  \\
          1    & \ H^1_m(G) \ & \ H^1_m(L) \ & \  0 & \   \cdots & \\
          0     & \ \mathbb{R}\   &\  \textup{L}^0(G/L)^G \ & \  \textup{L}^0((G/L)^2)^G  \ & \  \cdots & \\
    \quad\strut &   0  &  1  &  2   &  \cdots & q \strut \\};

%lines are counted starting from the top 

%arrows in the 3rd line
\draw[->](m-3-2.east) -- (m-3-3.west)node[midway,above ]{$0$};
\draw[->](m-3-3.east) -- (m-3-4.west)node[midway,above ]{$d^\rightarrow$};

%arrows in the 2nd line
\draw[->](m-4-2.east) -- (m-4-3.west)node[midway,above ]{$0$};
\draw[->](m-4-3.east) -- (m-4-4.west)node[midway,above ]{$d^\rightarrow$};

%arrows in the 1st line
\draw[->](m-5-2.east) -- (m-5-3.west)node[midway,above ]{$0$};
\draw[->](m-5-3.east) -- (m-5-4.west)node[midway,above ]{$d^\rightarrow$};

%arrows in the 1st line
\draw[->](m-6-2.east) -- (m-6-3.west)node[midway,above ]{$\delta$};
\draw[->](m-6-3.east) -- (m-6-4.west)node[midway,above ]{$\delta$};

\draw[thick] (m-1-1.east) -- (m-7-1.east) ;
\draw[thick] (m-7-1.north) -- (m-7-6.north) ;
\end{tikzpicture}
\caption{The first page ${}^{II}E_1$}\label{Page 1}
\end{figure}
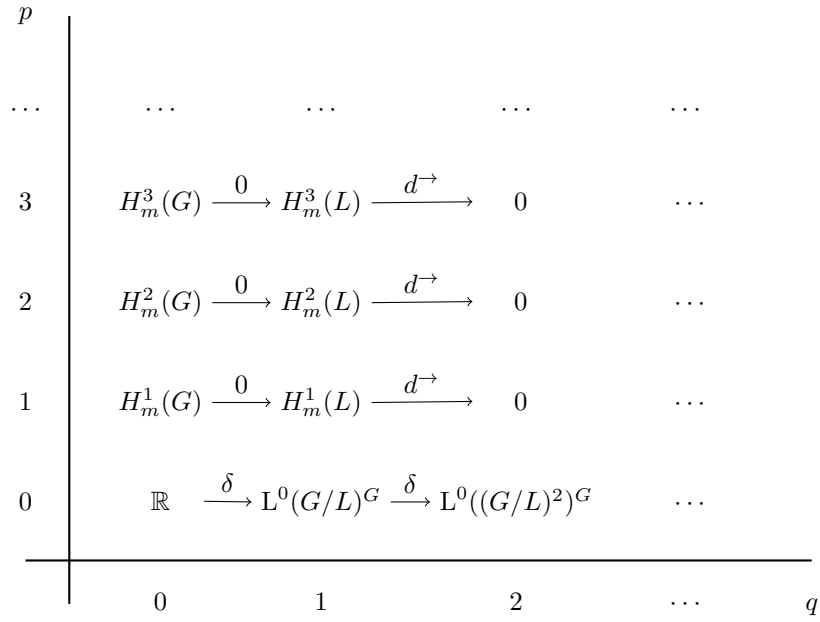

If we now move to the second page of the spectral sequence ${}^{II}E_2^{p,q}$, we immediately see that the first two columns survive and in the first row the cohomology groups of the action of $G$ on $G/L$ appear, namely $H^*_m(G \curvearrowright G/L)$. We refer to Figure \ref{Page 2} for the second page of the spectral sequence. 

\begin{figure}[!h]
\centering
\begin{tikzpicture}
  \matrix (m) [matrix of math nodes,
             nodes in empty cells,
             nodes={minimum width=9ex,
                    minimum height=9ex,
                    outer sep=-3pt},
             column sep=2ex, row sep=-1ex,
             text centered,anchor=center]{
         p      &         &          &          & \\
          \cdots & \cdots & \cdots & \cdots   & \cdots \\
          3    & \ H^3_m(G) \  & \ H^3_m(L) \ &\  0 & \ \cdots & \\
          2    &  \ H^2_m(G) \ &\  H^2_m(L) \  & \ 0 & \  \cdots &  \\
          1    & \ H^1_m(G) \ & \ H^1_m(L) \ &  \  0 & \   \cdots & \\
          0     & \ 0\   &\  0 \ & \  H^1_m(G \curvearrowright G/L)  \ &  \ \cdots & \\
    \quad\strut &   0  &  1  &  2  &  \cdots & q \strut \\};

%lines are counted starting from the top 

%arrows in the 3rd line
\draw[->](m-3-2) -- (m-4-4);
\draw[->](m-3-3) -- (m-4-5);

%arrows in the 2nd line
\draw[->](m-4-2) -- (m-5-4);
\draw[->](m-4-3) -- (m-5-5);

%arrows in the 1st line
\draw[->](m-5-2) -- (m-6-4);
\draw[->](m-5-3) -- (m-6-5);

\draw[thick] (m-1-1.east) -- (m-7-1.east) ;
\draw[thick] (m-7-1.north) -- (m-7-6.north) ;
\end{tikzpicture}
\caption{The second page ${}^{II}E_2$} \label{Page 2}
\end{figure}
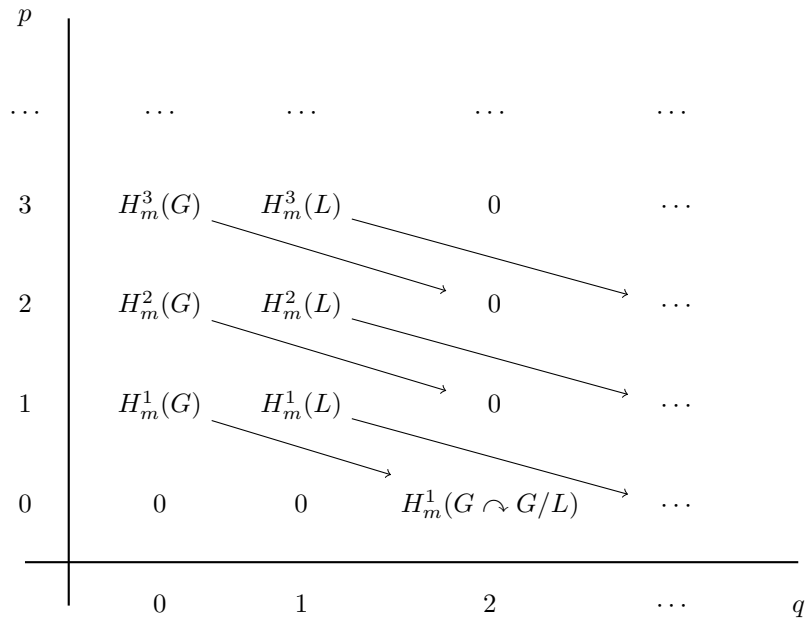

By Proposition \ref{prop collapse} we know that the spectral sequence ${}^IE_1^{p,q}$ degenerates immediately to zero. The spectral sequence ${}^{II}E_1^{p,q}$ must converge to the same limit, hence it must degenerate to zero \cite[Appendice A]{Guichardet}. If we denote by $d_p$ the differential corresponding to the $p$-th page ${}^{II}E_p$, the only way this can happen is that 

\begin{itemize}
\item we have an isomorphism $H^p_m(G \curvearrowright G/L) \cong H^p_m(G)$ for $p=0,1$, 
\item the differential $d_{p+1}$ gives an embedding of $H^{p}_m(L)$ into $H^{p+1}_m(G \curvearrowright G/L)$,
\item the differential $d_{p+2}$ gives an isomorphism between $H^{p+1}_m(G)$ and the quotient \mbox{$H^{p+1}_m(G \curvearrowright G/L)/d_{p+1}(H^p_m(L))$.} 
\end{itemize}
This concludes the proof of the existence of the desired short exact sequence. 

We now proceed with the compatibility with respect to inclusions. Take $L_0<L_1<P$ two closed subgroups with the desired compactness of 
stabilizers of pairs. The inclusion $L_0 \rightarrow L_1$ induces a projection $\pi:G/L_0 \rightarrow G/L_1$ which preserves the $G$-invariant measure classes. More precisely, if $\nu_0$ and $\nu_1$ are two quasi $G$-invariant measures on $G/L_0$ and $G/L_1$, respectively, then $\pi_*(\nu_0)$ and $\nu_1$ are equivalent. This implies that the natural induced map $L^0(G/L_1) \rightarrow L^0(G/L_0)$ determines a well-defined map at the level of bicomplexes
$$
\pi^{p,q}:L^0(G^{p+1},L^0((G/L_1)^q))^G \rightarrow L^0(G^{p+1},L^0((G/L_0)^q))^G.
$$
The map $\pi^{p,q}$ is compatible with both the differential $d^\uparrow$ and $d^\rightarrow$ and it descends to a map between the first pages of the spectral sequences, that is
$$
{}^{II}\pi^{p,q}_1:H^p_m(G,L^0((G/L_1)^q))^G \rightarrow H^p_m(G,L^0((G/L_0)^q))^G .
$$
When $p>1$ and $q=0$ the map ${}^{II}\pi^{p,q}$ boils down to the identity on the measurable cohomology of $G$. For $p>1$ and $q=1$, we have a commutative diagram
\begin{equation}\label{diagram inclusion}
\xymatrix{
L^0(G^{p+1})^{L_1} \ar[d]^{\mathrm{res}} \ar[rr]  &&  L^0(G^{p+1},L^0(G/L_1))^G \ar[d]^{\pi^{p,1}} \\
L^0(G^{p+1})^{L_0} \ar[rr] &&  L^0(G^{p+1},L^0(G/L_0))^G,
}
\end{equation}
where the top and the bottom map are the induction maps, the right arrow is the change of coefficients induced by $\pi$ and the left arrow is the restriction. By passing to the cohomology of Diagram \ref{diagram inclusion}, we get
\begin{equation}\label{diagram inclusion cohomology}
\xymatrix{
 H^p_m(L_1) \ar[d]^{\mathrm{res}} \ar[rr]  && H^p_m(G,L^0(G/L_1)) \ar[d]_{{}^{II}\pi^{p,1}_1}  \\
 H^p_m(L_0) \ar[rr] && H^p_m(G,L^0(G/L_0)),
}
\end{equation}
showing that ${}^{II}\pi^{p,1}_1$ is conjugated to the restriction from $L_1$ to $L_0$ via the Eckmann-Shapiro isomorphism. Finally when $p=0$, we get back 
$$
\pi^{0,q}_1:L^0((G/L_1)^q)^G \rightarrow L^0((G/L_0)^q)^G 
$$
which is well-defined because of the $G$-equivariance of $\pi$. The convergence to zero of both spectral sequences and the compatibility of ${}^{II}\pi^{p,q}$ with the differential of every degree imply that we must have Diagram \ref{diagram restriction identity}, and the theorem is proved. \end{proof}

\section{The derived series for the unipotent radical and the maximal split torus} \label{sec derived series torus}

In this section we will prove Corollary \ref{Thm N} and Corollary \ref{Thm A}. The main point will be to prove the compactness of a generic pair in the quotient. 

\begin{proof}[Proof of Corollary \ref{Thm N}] 
We need to show that for almost every pair of points in $G/N_k$ the stabilizer is compact. Since $N_k<P$ we have a $G$-equivariant projection 
$$
p:G/N_k \rightarrow G/P
$$
which preserves the measure classes. We say that two points in $G/N_k$ are \emph{generic} if they are mapped through the map $p$ to a pair of opposite points in $G/P$. Observe that the set of generic points has full measure in $G/N_k \times G/N_k$. 

Consider two generic points in $G/N_k$. Up to the $G$-action, we can suppose that our fixed pair is of the form $(N_k,gN_k) \in (G/N_k)^2$. Let $H=\mathrm{Stab}_G(N_k,gN_k)$ be the stabilizer of the pair. Since $H$ fixes $N_k<N$, it is a subgroup of $N$. The projection $p$ is $G$-equivariant, thus any element of $H$ fixes also the pair $(P,gP)$. Because of the genericity of $(N_k,gN_k)$, the pair $(P,gP)$ is opposite and hence it can be written as $(P,nw_0P)$ where $w_0$ is the longest element in the Weyl group and $n \in N$. As a consequence $H$ must be a subgroup of $nMAn^{-1}$ and hence
$$
H<N \cap nMAn^{-1}.
$$
Since the group on the right is compact, the claim follows. 
\end{proof}

We could define a notion of genericity for $G/A$ as in Corollary \ref{Thm N}, which would also have full measure in $G/A \times G/A$, but would not satisfy the needed compactness. Indeed, in that case, as $(P,w_0P)$ is a pair of opposite points, the pair $(A,w_0A)$ would be generic, but its stabilizer contains $A$. 

\begin{proof}[Proof of Corollary \ref{Thm A}]
Recall that the space $G/A$ surjects on pair of opposite points in $G/P$. More precisely, the surjection is given by 
$$
\varphi:G/A \longrightarrow G/P \times G/P , \ \  \varphi(gA):=(gP,gw_0P). 
$$
We define a pair of points $x,y$ in $G/A$ to be generic if, for $\varphi(x)=(x_1,x_2) \in G/P \times G/P$ and $\varphi(y)=(y_1,y_2) \in G/P \times G/P$, the pairs $(x_i,y_j)$ are opposite for all $i=1,2$ and $j=1,2$. Since for any such $i,j$ the pair $(x_i,y_j)$ is opposite for a subset of $G/A \times G/A$ of full measure, pair of generic points in $G/A$ indeed have full measure.

Finally observe that the stabilizer of the points $(x,y)$ is contained in the stabilizer of the generic triple $(x_1,x_2,y_1)$, which is compact by \cite[Proposition 5.1]{Monod}. 
\end{proof}

\begin{rem}
In the case of the derived series of the unipotent radical $N$, we know actually more. In fact $N_{k+1}<N_k$ and we can apply the second part of Theorem \ref{Main thm} to argue that there is a commutative diagram 
$$
\xymatrix{
H^{p-1}_m(N_k) \ar[rr] \ar[d]_{\mathrm{res}} && H^p_m(G \curvearrowright G/N_k) \ar[rr] \ar[d] && H^p_m(G) \ar[d]^{=}\\
H^{p-1}_m(N_{k+1}) \ar[rr]  && H^p_m(G \curvearrowright G/N_{k+1}) \ar[rr] && H^p_m(G) . \\
}
$$
\end{rem}

\section{The unipotent radical for $G=\mathrm{SL}(2,\mathbb{K})$}\label{sec unipotent radical SL2}

\subsection{Projection on $N$ and differentials} For $G=\mathrm{SL}(2,\mathbb{K})$ we fix the minimal parabolic subgroup of upper triangular matrices with coefficients in $\mathbb{K}$. The associated unipotent radical is given by
$$
N=\left\{  \left. \left( \begin{array}{cc} 1 & x \\ 0 & 1 \end{array} \right) \right| x \in \mathbb{K} \right\} \cong \mathbb{K}. 
$$
As a consequence, if we denote by $\mathfrak{n}^*$ the (real) dual of the Lie algebra associated to $N$, we have that 
$$
H^1_m(N) \cong \mathfrak{n}^* \cong 
\begin{cases}
\mathbb{R}^2, &  \text{for $\mathbb{K}=\mathbb{C}$} \\
\mathbb{R},     &  \text{for $\mathbb{K}=\mathbb{R}$} \\ 
\end{cases}
$$
In a similar way, one can see quite easily that the quotient $G/N$ is diffeomorphic to the punctured plane $\mathbb{K}^2 \setminus \{0\}$. In fact the latter space is $G$-homogeneous and, if we denote by $\mathcal{E}=\{e_1,e_2\}$ the canonical basis of $\mathbb{K}^2$, the stabilizer of the vector $e_1$ is $\mathrm{Stab}_G(e_1)=N$.

Before moving on with our investigation, recall that by the Iwasawa decomposition, we can write any element $g \in G$ in a unique way as the product $g=nak$, where $a \in A$ lies in the maximal split torus of real diagonal matrices with positive entries, $n \in N$ and $k \in K$, where either $K=\mathrm{SO}(2)$ for $\mathbb{K}=\mathbb{R}$, or $K=\mathrm{SU}(2)$, when $\mathbb{K}=\mathbb{C}$. In virtue of such decomposition, we define the $N$-\emph{projection} as the map 
$$
\pi_N:G \rightarrow N, \ \ g=nak \mapsto n.
$$
In what follows we will identify $N$ with the field $\mathbb{K}$.

\begin{lem}\label{lemma N projection}
Given 
$$
g=\left(
\begin{array}{cc}
a_{11} & a_{12} \\
a_{21} & a_{22}
\end{array}
\right) \in \mathrm{SL}(2,\mathbb{K}),
$$  
the $N$-projection of $g$ is given by 
\begin{equation}\label{equation N projection}
\pi_N(g)=\frac{\langle (a_{11},a_{12}), (a_{21},a_{22}) \rangle}{|a_{21}|^2 + |a_{22}|^2} \in \mathbb{K} \cong \mathfrak{n} \cong N,
\end{equation}
where $\langle \cdot, \cdot \rangle$ is either the scalar product, when $\mathbb{K}=\mathbb{R}$, or the Hermitian product, when $\mathbb{K}=\mathbb{C}$. 
\end{lem}

\begin{proof}
We denote by $g^*$ the $\mathbb{K}$-adjoint matrix to $g$, that is either its transpose, in the real case, or its conjugated transpose, in the complex case. The map 
$$
p:G \rightarrow X, \ \ p(g)=gg^*:=s
$$
defines a natural projection on the symmetric space $X$ associated to $G$. Recall that such space is the set of positive definite either symmetric matrices, for $\mathbb{K}=\mathbb{R}$, or Hermitian matrices, for $\mathbb{K}=\mathbb{C}$, with determinant equal to one.

 The Iwasawa decomposition of $g=nak$ leads to the following chain of equalities
\begin{equation}\label{equation Iwasawa}
s=gg^*=(nak)(nak)^*=nakk^*a^*n^*=naa^*n^*,
\end{equation}
where the matrix $k$ disappeared because it is either orthogonal or unitary. A direct computation shows that Equation \eqref{equation Iwasawa} is equivalent to Equation \eqref{equation N projection}.
\end{proof}

%We now move to the contracting homotopies required by Lemma \ref{lemma explicit injection}. For the cocomplexes $(C^{*,q},d^\uparrow)$, when $q \geq 2$, those homotopies can be easily defined as follows
%$$
%h^{p-1,q}:C^{p,q} \rightarrow C^{p-1,q}
%$$
%\begin{align*}
%h^{p-1,q}(f)(g_0,\dots, g_{p-1})(&e_1, \lambda e_2, v_3, v_4,\dots, v_q):= \\
%             f(e,g_0,\dots, g_{p-1})(&e_1, \lambda e_2, v_3, v_4,\dots, v_q),
%\end{align*}
%where $\lambda \in \mathbb{K}^*$. The previous definition gives us back a measurable function defined an a dense subset of $G^p$: although $f$ is only defined on a full measure subset of $G^{p+1}$, the evaluation at a tuple of the form $(e,g_0,\ldots,g_{p-1})$ makes sense on a subset of full measure of $G^{p+1}$, thanks to the $G$-invariance of $f$. In a similar way, the evalution on the tuple $(e_1,\lambda e_2,v_3,v_4,\ldots,v_q)$ makes sense by the transitivity of the $G$-action and by the fact that we can always write the first two generic vectors in the form $(e_1,\lambda e_2)$. The resulting cochain is automatically $G$-invariant and $h^{p,q}$ verify the desired contracting homotopy equation. 

We conclude the section by computing the differential $d^\rightarrow:C^{p,1} \rightarrow C^{p,2}$. Before doing that, notice that the notion of genericity for pairs of points in $\mathbb{K}^2 \setminus \{ 0 \}$ boils down to the usual linear independence. Additionally the $G$-action is not transitive on linearly independent pairs of $\mathbb{K}^2 \setminus \{ 0 \}$, and the space of orbits is diffeomorphic to $\mathbb{K}^*$. In fact, in the $G$-orbit of a pair of linearly independent vectors $(u,v)$, there is a preferred pair given by $(e_1,d_{u,v}e_2)$, where $d_{u,v}$ is the determinant of the matrix $(u|v)$ with columns given by $u$ and $v$. If $u=(u_1,u_2)$ and $v=(v_1,v_2)$ are linearly independent, we can define 
\begin{equation}\label{equation guv}
g_{u,v}:=
\left(
\begin{array}{cc}
u_1 & \frac{v_1}{d_{u.v}} \\
u_2 & \frac{v_2}{d_{u,v}}
\end{array}
\right) \in \mathrm{SL}(2,\mathbb{K})
\end{equation}
and one has immediately that $g_{u,v}(e_1,d_{u,v}e_2)=(u,v)$. Notice that $g_{u,v}$ is uniquely defined since the stabilizer $\mathrm{Stab}_G(e_1,\lambda e_2)$ is trivial for every $\lambda \neq 0$. 

As a consequence the following isomorphism holds
$$
C^{p,2} = L^0(G^{p+1},L^0((G/N)^2))^G \cong L^0(G^{p+1} \times \mathbb{K}^*) . 
$$

We finally define 
$$
\delta_\lambda:=
\left(
\begin{array}{cc}
0 & -\frac{1}{\lambda} \\
\lambda & 0 
\end{array}
\right), 
$$
so that $\delta_\lambda e_1=\lambda e_2$. It is clear that $\delta_\lambda$ is not unique, since we can modify it by right multiplication by any element in $N$. 

\begin{lem}\label{lemma SL2 differential}
The differential 
$$
d^\rightarrow:L^0(G^{p+1})^N \cong C^{p,1} \rightarrow L^0(G^{p+1} \times \mathbb{K}^*) \cong C^{p,2}
$$
is given by 
$$
d^\rightarrow (\beta)(g_0,\ldots,g_p)(\lambda)=(-1)^{p+1}[\beta(\delta_{-\lambda}g_0,\ldots,\delta_{-\lambda}g_p)-\beta(g_0,\ldots,g_p)]
$$
\end{lem}

At a first sight the right-hand side may depend on the particular choice of the element $\delta_\lambda$, but this is not true. In fact, suppose that we substitute $\delta_\lambda$ by multiplying on right by an element in $N$. Its inverse is multiplied on the left and the $N$-invariance of $\beta$ allows to get rid of this issue. 

\begin{proof}
A cochain $\beta \in L^0(G^{p+1})^N$ determines a cochain $\overline{\beta} \in C^{p,1}=L^0(G^{p+1},L^0(G/N))^G$ via the induction map 
$$
\overline{\beta}(g_0,\ldots,g_p)(he_1):=\beta(h^{-1}g_0,\ldots,h^{-1}g_p) .
$$
Fixing our preferred generic pair of the form $(e_1,\lambda e_2)$, we have that the differential $d^\rightarrow \beta \in L^0(G^{p+1} \times \mathbb{K}^*)$ can be written as
$$
d^\rightarrow\beta(g_0,\ldots,g_p)(\lambda):=d^\rightarrow \overline{\beta}(g_0,\ldots,g_p)(e_1,\lambda e_2),
$$
and the latter evaluation makes sense because of the $G$-invariance of the coboundary $d^\rightarrow \overline{\beta}$. 
Using the definition of $d^\rightarrow$ and the fact $\delta_\lambda e_1=\lambda e_2$ we obtain 
\begin{align*}
d^\rightarrow \overline{\beta}(g_0,\ldots,g_p)(e_1,\lambda e_2)=&(-1)^{p+1}[\overline{\beta}(g_0,\ldots,g_p)(\lambda e_2)  - \overline{\beta}(g_0,\ldots,g_p)(e_1)]\\
=&(-1)^{p+1}[\beta(\delta_{-\lambda}g_0,\ldots,\delta_{-\lambda}g_p)-\beta(g_0,\ldots,g_p)].
\end{align*} \end{proof}

\subsection{Proof of Theorem \ref{Thm SL2 N}} We consider a $\mathbb{R}$-linear functional $\alpha_{\mathfrak{n}} \in \mathfrak{n}^*$. Since $\mathfrak{n}$ is diffeomorphic to $N$ through the exponential map, we can build an inhomogeneous cocycle $\overline{\alpha}:N \rightarrow \mathbb{R}$ as follows
$$
\overline{\alpha}(n):=\alpha_{\mathfrak{n}}(\log n)
$$
and it homogeneized variant $\alpha:N^2 \rightarrow \mathbb{R}$ is given by 
$$
\alpha(n_0,n_1):=\overline{\alpha}(n_0^{-1}n_1)=\overline{\alpha}(n_1)-\overline{\alpha}(n_0).
$$
The latter cocycle can be easily extended to the whole group $G$ by precomposing it with the $N$-projection defined by Lemma \ref{lemma N projection}, that is
$$
\alpha_G:G^2 \rightarrow \mathbb{R}, \ \ \alpha_G(g_0,g_1)=\alpha(\pi_N(g_0),\pi_N(g_1)) . 
$$
By construction, the cocycle $\alpha_G$ is $N$-invariant and hence it is an element of $L^0(G^2)^N \cong C^{1,1}$. Additionally its cohomology class in $H^1(C^{1,1},d^\uparrow) \cong H^1_m(N) \cong \mathfrak{n}^*$ coincides with the homomorphism $\alpha_{\mathfrak{n}}$. 

To compute the resulting cocycle in the cohomology of the action, we need to follow $\alpha_G$ in the diagram below:
\begin{equation}
\xymatrix{
\alpha_G \in C^{1,1} \ar[rr] && d^\rightarrow\alpha_G=d^\uparrow \beta \in C^{1,2}  && \\
&& \beta \in C^{0,2} \ar[rr] \ar[u] && d^\rightarrow \beta=:\omega \in C^{0,3} \\
}
\end{equation}

\subsubsection*{Computation of $\beta$} We start computing the right differential of $\alpha_G$. By Lemma \ref{lemma SL2 differential} we have 
$$
d^\rightarrow \alpha_G(g_0,g_1)(\lambda)=\alpha_G(\delta_{-\lambda}g_0,\delta_{-\lambda}g_1)-\alpha_G(g_0,g_1).
$$

If we set 
$$
\beta(g)(\lambda):=\alpha_G(\delta_{-\lambda},\delta_{-\lambda}g)-\alpha_G(e,g),
$$
it is immediate to verify that $d^\uparrow \beta=d^\rightarrow \alpha_G$, as desired. An alternative expression of $\beta$ is given by 
\begin{equation}\label{equation beta g}
\beta(g)(\lambda)=\alpha(\pi_N(\delta_{-\lambda}g))-\alpha(\pi_N(g))
\end{equation}
since both the $N$-projection of $\delta_{-\lambda}$ and $e$ are trivial. In what follows it will be important to write an explicit expression of the cocycle $\overline{\beta}$, obtained by induction, evaluated on $\overline{\beta}(e)(u,v)$, where $(u,v)$ are linearly independent. Thanks to the definition of the matrix $g_{u,v}$ in Equation \ref{equation guv}, we can argue from Equation \ref{equation beta g} to get that
\begin{align*}
\overline{\beta}(e)(u,v)=&\overline{\beta}(e)(g_{u,v}(e_1,d_{u,v}e_2))=\beta(g_{u,v}^{-1})(d_{u,v})\\
=&\alpha(\pi_N(\delta_{-d_{u,v}}g_{u,v}^{-1}))-\alpha(\pi_N(g^{-1}_{u,v}))
\end{align*}
Exploting Lemma \ref{lemma N projection} we can compute explicitly 
$$
\pi_N(\delta_{-d_{u,v}}g^{-1}_{u,v})=\frac{\langle \frac{u}{d_{u,v}} , v \rangle}{\|v\|^2}, \ \  \pi_N(g_{u,v}^{-1})=\frac{-\langle \frac{v}{d_{u,v}},u \rangle}{\|u\|^2}
$$
By the fact that $\alpha$ is $\mathbb{R}$-linear and viewing $\mathbb{K} \cong \mathfrak{n} \cong N$ we obtain 
\begin{equation}\label{equation beta explicit}
\overline{\beta}(e)(u,v)=\alpha_{\mathfrak{n}}\left(\langle  \frac{u}{d_{u,v}} , \frac{v}{\|v\|^2} \rangle \right)+\alpha_{\mathfrak{n}}\left( \langle \frac{v}{d_{u,v}} , \frac{u}{\|u\|^2} \rangle \right).
\end{equation}

\subsubsection*{Computation of $\omega$} Since in this case the computation has only two steps, to get the desired cocycle, it will be sufficient fo evaluate $\omega(e)$ on a generic triple. Using Equation \ref{equation beta explicit}, the computation is immediate:
\begin{align*}
\omega(e)(v_0,v_1,v_2)&=d^\rightarrow \beta(e)(v_0,v_1,v_2)\\
&=-\beta(e)(v_1,v_2)+\beta(e)(v_0,v_2)-\beta(e)(v_0,v_1)\\
&=\sum_{i < j} (-1)^{i+j} \left[ \alpha_{\mathfrak{n}}\left(\langle  \frac{v_i}{d_{v_i,v_j}} , \frac{v_j}{\|v_j\|^2} \rangle \right)+\alpha_{\mathfrak{n}}\left( \langle \frac{v_j}{d_{v_i,v_j}} , \frac{v_i}{\|v_i\|^2} \rangle \right)\right],
\end{align*}
and the theorem is proved. 

\section{The maximal split torus case for $\mathrm{SL}(2,\mathbb{R})$} \label{sec maximal split torus SL2}

\subsection{Projection on A and differentials} For $G=\mathrm{SL}(2,\mathbb{R})$, we fix as $A$ the subgroup of diagonal matrices with positive entries, that is
$$
A=\left\{ \left. a_\lambda:=\left(
\begin{array}{cc}
\lambda & 0 \\
0 & \frac{1}{\lambda}
\end{array}
\right)
\right| 
\lambda \in \mathbb{R}_{>0}
 \right\} . 
$$
If we denote by $\mathfrak{a}$ the Lie algebra associated to $A$, and by $\mathfrak{a}^*$ its dual, we have that
$$
H^1_m(A) \cong \mathfrak{a}^\ast \cong \mathbb{R}. 
$$
The quotient $G/A$ parametrizes a subset of full measure in $\mathbb{P}^1(\mathbb{R})^2$, namely the one of distinct ordered pairs. We choose as basepoint in $G/A$ the pair $(\infty,0)$, for which we clearly have $\mathrm{Stab}_G(\infty,0)=A$. 

Again by the Iwasawa decomposition, we can factor any element $g \in G$ as the product $g=nak$, where $n \in N$ is an upper triangular unipotent matrix, $a \in A$ and $k \in \mathrm{SO}(2)$. The $A$-\emph{projection} of $g \in G$ is defined as follows
$$
\pi_A:G \rightarrow A, \ g=nak \mapsto a. 
$$
The same strategy of the proof of Lemma \ref{lemma N projection} can be used to prove the following

\begin{lem}\label{lemma A projection}
Given 
$$
g=\left(
\begin{array}{cc}
a_{11} & a_{12} \\
a_{21} & a_{22} 
\end{array}
\right) \in \mathrm{SL}(2,\mathbb{R}),
$$
the $A$-projection of $g$ is
$$
\pi_A(g)=a_\lambda:=\left(
\begin{array}{cc}
\lambda & 0 \\
0 & \frac{1}{\lambda}
\end{array}
\right), \ \lambda=\frac{1}{\sqrt{a_{21}^2+a_{22}^2}}.
$$
\end{lem}

We conclude by computing explicitly the differential $d^\rightarrow:C^{p,1} \rightarrow C^{p,2}$. The notion of genericity for two different pairs $(x_1,x_2)$ and $(y_1,y_2)$ requires that the $4$ points $x_1,x_2,y_1,y_2$ are all distinct. In this context, it is well-known that the triple $(x_1,x_2,y_1)$ must be either positively oriented or negatively
oriented. In the first case, in the $G$-orbit of the $4$-tuple $(x_1,x_2,y_1,y_2)$ there exists a unique representative of the form $(\infty,0,1,b)$, where $b$ 
is the \emph{cross ratio}. Analogously, in the negatively oriented case we can choose as representative of the $G$-orbit the $4$-tuple $(\infty,0,-1,b)$. 
Thus we can parametrize the $G$-orbits on a full measure subset of $(G/A)^2$ by the orientation and the cross ratio, namely $(\pm 1, b)$.

Finally, given an ordered pair $(x,y)$, we define the matrix
$$
g_{x,y}:=
\left(
\begin{array}{cc}
x & -\frac{y}{y-x} \\
1 & -\frac{1}{y-x}
\end{array}
\right) \in \mathrm{SL}(2,\mathbb{R}),
$$
which satisfies $g_{x,y}(\infty,0)=(x,y)$. It is worth noticing that $g_{x,y}$ is not uniquely defined: in fact we can multiply $g_{x,y}$ on the right by any element of the stabilizer $\mathrm{Stab}_G(\infty,0)=A$. 

\begin{lem}\label{lemma explicit differential A}
The differential
$$
d^\rightarrow:L^0(G^{p+1})^A \cong C^{p,1} \rightarrow C^{p,2}
$$
is given by 
\begin{equation}\label{eq diff A}
d^\rightarrow(\beta)(g_0,\ldots,g_p)(\pm 1,b)=(-1)^{p+1}[\beta(g_{\pm 1,b}^{-1}g_0,\ldots,g_{\pm 1,b}^{-1}g_p)-\beta(g_0,\ldots,g_p)].
\end{equation}
\end{lem} 

Notice that the right-hande side of Equation \eqref{eq diff A} is well-defined thanks to the $A$-invariance of the cocycle $\beta$. In fact $g_{\pm 1,b}$ is defined up to right multiplication by an element in $A$. As a consequence, its inverse is defined up to left $A$-multiplication, but the $A$-invariance of $\beta$ allows to get rid of this issue. 

\begin{proof}
The proof is completely analogous to the one of Lemma \ref{lemma SL2 differential}. Take a cochain $\beta \in L^0(G^{p+1})^A$. In virtue of the Eckmann-Shapiro induction map, we can define $\overline{\beta} \in C^{p,1} = L^0(G^{p+1},L^0(G/A))^G$ as follows
$$
\overline{\beta}(g_0,\ldots,g_p)(h(\infty,0)):=\beta(h^{-1}g_0,\ldots,h^{-1}g_p). 
$$ 
We now consider a pair of ordered pairs $((\infty,0),(\pm 1,b))$, where we have the right to take those three initial points thanks to the $G$-transitivity. The right differential is given by 
$$
d^\rightarrow(g_0,\ldots,g_p)(\pm 1,b):=d^\rightarrow \overline{\beta}(g_0,\ldots,g_p)((\infty,0),(\pm 1,b)).
$$
Since $g_{\pm 1,b}(\infty,0)=(\pm 1,b)$, the definition of $d^\rightarrow$ leads to 
\begin{align*}
d^\rightarrow \overline{\beta}(g_0,\ldots,g_p)((\infty,0),(\pm 1,b))=&(-1)^{p+1}[\overline{\beta}(g_0,\ldots,g_p)(\pm 1,b)  - \overline{\beta}(g_0,\ldots,g_p)(\infty,0)]\\
=&(-1)^{p+1}[\beta(g_{\pm 1,b}^{-1}g_0,\ldots,g_{\pm 1,b}^{-1}g_p)-\beta(g_0,\ldots,g_p)].
\end{align*}
\end{proof}

\subsection{Proof of Theorem \ref{Thm SL2 A}} Let $\alpha_{\mathfrak{a}} \in \mathfrak{a}^*$ be a linear functional. We can define 
$$
\overline{\alpha}:A \rightarrow \mathbb{R}, \ \overline{\alpha}(a):=\alpha_{\mathfrak{a}}(\log a),
$$
and from it we can define the homogeneized cocycle 
$$
\alpha:A^2 \rightarrow \mathbb{R}, \ \alpha(a_0,a_1):=\overline{\alpha}(a_1)-\overline{\alpha}(a_0). 
$$
We finally extend $\alpha$ to a $A$-invariant cocycle $\alpha_G:G^2 \rightarrow \mathbb{R}$ by precomposing with the $A$-projection of Lemma \ref{lemma A projection}, namely
$$
\alpha_G(g_0,g_1):=\alpha(\pi_A(g_0),\pi_A(g_1)).
$$
In this way we obtain an element $\alpha_G \in L^0(G^2)^A \cong C^{1,1}$ whose cohomology class in $H^1(C^{1,1},d^\uparrow) \cong H^1_m(A) \cong \mathfrak{a}^*$ is exactly the homomorphism $\alpha_\mathfrak{a}$. 

We need to follow $\alpha_G$ along the diagram below: 
\begin{equation}\label{equation snake A} 
\xymatrix{
\alpha_G \in C^{1,1} \ar[rr] && d^\rightarrow\alpha_G=d^\uparrow \beta \in C^{1,2}  && \\
&& \beta \in C^{0,2} \ar[rr] \ar[u] && d^\rightarrow \beta=:\omega \in C^{0,3} \\
}
\end{equation}

\subsubsection*{Computation of $\beta$} By Lemma \ref{lemma explicit differential A} we have that 
$$
d^\rightarrow \alpha_G(g_0,g_1)(\pm 1, b)=\alpha_G(g^{-1}_{\pm 1,b}g_0,g^{-1}_{\pm 1,b}g_1)-\alpha_G(g_0,g_1).
$$
If we set
$$
\beta(g)(\pm 1,b):=\alpha_G(g_{\pm 1,b}^{-1},g_{\pm 1,b}^{-1}g)-\alpha_G(e,g). 
$$
we get that $d^\uparrow \beta=d^\rightarrow \alpha_G$, as desired. 

Now we would like to obtain an expression of image of $\beta$ through induction when $g=e$ and for any pair of generic pairs $((x_1,x_2),(y_1,y_2))$. In order to do that, we will distinguish two possible cases, since we may have that the triple $(x_1,x_2,y_1)$ is either positively oriented or negatively oriented. Assume that the triple $(x_1,x_2,y_1)$ is positively oriented. If we set 
$$
c=\sqrt{\frac{y_1-x_2}{(x_2-x_1)(y_1-x_1)}}
$$
and we define
\begin{equation}\label{equation gxyz}
g_{x_1,x_2,y_1}:=
\left(
\begin{array}{cc}
cx_1 & -\frac{x_2}{c(x_2-x_1)}\\
c & -\frac{1}{c(x_2-x_1)} 
\end{array}
\right) \in \mathrm{SL}(2,\mathbb{R}),
\end{equation}
then we have $g_{x_1,x_2,y_1}(\infty,0,1,b)=(x_1,x_2,y_1,y_2)$. Here $b=[x_1,x_2,y_1,y_2]$ is the cross ratio. As a consequence, we can write
\begin{align}\label{beta positive}
\beta(e)((x_1,x_2),(y_1,y_2))&=\beta(e)(g_{x_1,x_2,y_1}((\infty,0),(1,b)))\\
&=\beta(g_{x_1,x_2,y_1}^{-1})((\infty,0),(1,b)) \nonumber \\
&=\alpha_G(g_{1,b}^{-1},g_{1,b}^{-1}g_{x_1,x_2,y_1}^{-1})-\alpha_G(e,g_{x_1,x_2,y_1}^{-1})  \nonumber \\
&=\frac{1}{2}\alpha_{\mathfrak{a}}\left(\log \left(\frac{2(x_1^2 + 1)(x_2 - y_1)^2}{(y_1^2 + 1)(x_1 - x_2)^2}\right) \right), \nonumber 
\end{align}
where we exploited the definition of $\alpha_G$ and the projection given in Lemma \ref{lemma A projection} to move from the third line to the last one. 

We now move to the case where $(x_1,x_2,y_1)$ is negatively oriented. This time we need to set
$$
c=\sqrt{\frac{x_2-y_1}{(x_1-x_2)(x_1-y_1)}}
$$
and we define $g_{x_1,x_2,y_1}$ as in Equation \eqref{equation gxyz}. We have that $g_{x_1,x_2,y_1}(\infty,0,-1,b)=(x_1,x_2,y_1,y_2)$ and we can write
\begin{align}\label{beta negative}
\beta(e)((x_1,x_2),(y_1,y_2))&=\beta(e)(g_{x_1,x_2,y_1}((\infty,0),(-1,b)))\\
&=\beta(g_{x_1,x_2,y_1}^{-1})((\infty,0),(-1,b)) \nonumber \\
&=\alpha_G(g_{-1,b}^{-1},g_{-1,b}^{-1}g_{x_1,x_2,y_1}^{-1})-\alpha_G(e,g_{x_1,x_2,y_1}^{-1})  \nonumber \\
&=\frac{1}{2}\alpha_{\mathfrak{a}}\left(\log \left(\frac{2(x_1^2 + 1)(x_2 - y_1)^2}{(y_1^2 + 1)(x_1 - x_2)^2}\right) \right) \nonumber 
\end{align}
Since the evaluations of $\beta$ in both Equation \eqref{beta positive} and Equation \eqref{beta negative} are the same, we have that $\beta(e)((x_1,x_2),(y_1,y_2))$ is given by Equation \eqref{beta positive} for any pair of generic pairs. 

\subsubsection*{Computation of $\omega$} Take any triple of pairwise generic pairs $((x_1,x_2),(y_1,y_2),(z_1,z_2))$, that means that the six points in the boundary are all distinct. The evaluation of $\omega$ on the neutral element and the chosen triple is given by
\begin{align*}
\omega(e)((x_1,x_2),(y_1,y_2),(z_1,z_2))&=d^\rightarrow \beta(e)((x_1,x_2),(y_1,y_2),(z_1,z_2))\\
&=-\beta(e)((y_1,y_2),(z_1,z_2))+\beta(e)((x_1,x_2),(z_1,z_2))-\beta(e)((x_1,x_2),(y_1,y_2))\\
&=-\frac{1}{2}\alpha_\mathfrak{a}\left(\log\left(\frac{(z_1-y_2)(y_1-x_2)}{(z_1-x_2)(y_2-y_1)}\right)\right)\\
&=-\frac{1}{2}\alpha_\mathfrak{a}(\log(|[x_2,y_1,y_2,z_1]-1|),
\end{align*}
which finishes the proof.

\bibliographystyle{alpha}

\end{document}